\documentclass[12pt]{amsart}
 
\usepackage {amsfonts, amssymb, amscd, amsthm, amsmath}
\usepackage{geometry}
\usepackage[bookmarksopen, naturalnames]{hyperref}
\usepackage{tikz}
\newtheorem{Thm}{Theorem}[section]
\newtheorem{theorem}[Thm]{Theorem}
\newtheorem*{theorem*}{Theorem}
\newtheorem*{remark*}{Remark}

\newtheorem*{obs*}{Observation}
\newtheorem*{prob*}{Problem}
\newtheorem{corollary}[Thm]{Corollary}

\newcommand{\co}{\mathbb{C}}
\newcommand{\N}{\mathbb{N}}

\newcommand{\D}{\mathbb{D}}
\newcommand{\Dh}{\widehat{\mathbb{D}}}
\newcommand{\T}{\mathbb{T}}

\title[Point spectrum for a class of truncated Toeplitz operators]{Point spectrum and hypercyclicity problem \\
for a class of truncated Toeplitz operators}

\author[A. Baranov, A. Lishanskii]{Anton Baranov, Andrei Lishanskii}

\address{Anton Baranov\\
	Department of Mathematics and Mechanics, St. Petersburg State University, 28, Universitetskii prosp., St. Petersburg, 198504, Russia}
\email{anton.d.baranov@gmail.com}

\address{Andrei Lishanskii\\
St. Petersburg University, 7/9, Universitetskaya nab., St. Petersburg, 199034, Russia }
	\email{lishanskiyaa@gmail.com}

\thanks{The work of A. Lishanskii in Section 2 was performed at the Saint Petersburg Leonhard Euler International Mathematical
Institute and supported by the Ministry of Science and Higher Education of the Russian Federation (agreement no. 075-15-2022-287).
The results of Section 3 were obtained with the support of the Russian Science Foundation project 19-11-00058P}

\keywords{hypercyclic operator, Toeplitz operator, model space, truncated Toeplitz operator}

\subjclass{47A16, 47B35, 30H10}

\begin{document}

\begin{abstract}
In this note we discuss an open problem whether a truncated Toeplitz operator
on a model space can be hypercyclic.
We compute point spectrum and eigenfunctions for a class of 
truncated Toeplitz operators with polynomial analytic and antianalytic parts. 
We show that, for a class of model spaces, truncated Toeplitz operators with symbols of the form
$\Phi(z) =a \bar{z} +b + cz$, $|a| \ne |c|$, have complete sets of eigenvectors, and,
in particular, are not hypercyclic.  
\end{abstract}

\maketitle
\sloppy

\section{Introduction}

A continuous linear operator $T$ on a separable Banach (or Frech\'et) space 
$X$  is said to be  \textit{hypercyclic} if 
there exists $x \in X$ such that the set 
$\{T^n x:\, n\in\mathbb{N}_0\}$ is dense in $X$
(here $\mathbb{N}_0 = \{0,1,2, \dots\}$). 

The first example
of a hypercyclic operator in a Banach space setting was given by 
S. Rolewicz \cite{rol} who showed that the operator $\alpha B$, where $B$
is the backward shift and $\alpha$ is an arbitrary complex number with $|\alpha|>1$,
is hypercyclic on $\ell^p$, $1\le p<\infty$. It was shown that many important classes of operators 
have this property. Among basic examples of hypercyclic operators were Toeplitz operators with antianalytic symbols.

\subsection{Hypercyclic Toeplitz operators}
As usual, $\mathbb{D}$ and $\mathbb{T}$ will denote the unit disk and 
the unit circle, respectively, and $H^2$ will stand for the Hardy space in $\D$.
Recall that for a function $\psi \in L^\infty(\mathbb{T})$ the {\it Toeplitz operator}
$T_\psi: H^2 \to H^2$ with the symbol $\psi$  is defined as 
$T_\psi f = P_+(\psi f)$, where $P_+$ is the orthogonal projection from
$L^2(\mathbb{T})$ onto  $H^2$. The following theorem, due to 
G. Godefroy and J. Shapiro \cite{gosh}, describes all hypercyclic Toeplitz operators 
with antianalytic symbols (i.e., $T_{\overline{\varphi}}$ where $\varphi \in H^\infty$).

\begin{theorem*} {\rm (G. Godefroy, J. Shapiro, 1991)}
The Toeplitz operator $T_{\overline{\varphi}}$, where $\varphi \in H^\infty$, 
is hypercyclic if and only if $\varphi(\mathbb{D}) \cap \mathbb{T} \ne \emptyset$.
\end{theorem*}

It is well known that every Cauchy kernel $k_\lambda(z) = (1- \bar \lambda z)^{-1}$, $|\lambda|<1$,  
is an eigenvector for an  antianalytic Toeplitz operator $T_{\overline{\varphi}}$ corresponding to the eigenvalue
$\overline{\varphi(\lambda)}$. Then sufficiency of the condition 
$\varphi(\mathbb{D}) \cap \mathbb{T} \ne \emptyset$ for hypercyclicity follows immediately from the Godefroy--Shapiro
Criterion (see \cite{gosh}, \cite[Corollary 1.10]{bm} or \cite[Theorem 3.1]{gp}).

Analytic Toeplitz operators $T_\varphi$, $\varphi \in H^\infty$, are simply multiplication operators. It is clear that they 
cannot be hypercyclic.

In  \cite{sh} S. Shkarin posed the problem to describe hypercyclic Toeplitz operators 
in terms of their symbols. This seems to be a difficult and, at the moment, widely open problem. 
Shkarin gave a necessary and sufficient condition for hypercyclicity in the case when the symbol is of the form
$\Phi(z) =a\overline{z} +b +cz$
(i.e., with tridiagonal matrix): $T_\Phi$ is hypercyclic if and only if $|a|> |c|$ 
and
\begin{equation}
\label{spe}
(\co \setminus \Phi(\D))  \cap \D \ne \emptyset, \qquad (\co \setminus \Phi(\D)) \cap \Dh \ne \emptyset.
\end{equation}
Here $\Dh = \co\setminus \overline{\D}$ and $\Phi$ is extended to $\co$ as 
$\Phi(z) =\frac{a}{z} +b +cz$.

Hypercyclicity of Toeplitz operators with the symbols of the form
$$
\Phi(z) = R(\overline{z}) +\varphi(z),
$$
where $R$ is a polynomial or a rational function with poles outside $\overline{\D}$ and $\varphi\in H^\infty$,
was studied in \cite{abcl, barl}. It turned out that the valence of the meromorphic continuation of the symbol $\Phi$, 
that is, $\Phi(z)= R\big(\frac{1}{z}\big)+\varphi(z)$, is crucial for hypercyclicity. Also it was shown that
the hypercyclicity problem is closely related to cyclicity of certain families for some associated analytic Toeplitz operator. 
A complete description of hypercyclic Toeplitz operators in this class is not known, however, there is a number of conditions sufficient
for hypercyclicity. For the operators with symbols of the form $\Phi(z) = a\overline{z}  +\varphi(z)$ known necessary
and sufficient conditions almost meet: 

\begin{theorem*} {\rm (\cite{barl})}
Let $\Phi(z) = \frac{a}{z} +\varphi(z)$, where $\varphi \in H^\infty$, satisfy \eqref{spe}.

1. If $T_\Phi$ is hypercyclic, then $\Phi$ is univalent in $\D$\textup;

2. If $\varphi\in A(\D)$ \textup(disk-algebra\textup) and $\Phi$ is injective in $\overline{\D}$, then 
$T_\Phi$ is hypercyclic.
\end{theorem*}

Thus, the difference is only in the injectivity on the boundary. Also it should be noted that the second part of \eqref{spe}
(i.e. $(\co \setminus \Phi(\D)) \cap \Dh \ne \emptyset$)
is necessary for hypercyclicity of $T_\Phi$, while 
for the first part only a weaker variant $(\co \setminus \Phi(\D))  \cap \T \ne \emptyset$ is known to be necessary.
\medskip

\subsection{Truncated Toeplitz operators}
Let $S^*f =T_{\bar z} f = \frac{f(z) - f(0)}{z}$ be the backward shift in $H^2$.
By the classical theorem of Arne Beurling, any $S^*$-invariant subspace is of the form $K_\theta = H^2\ominus\theta H^2$
for some inner function $\theta$ in $\D$. Subspaces $K_\theta$ are often referred to as {\it model spaces} due to 
their role in Nagy--Foia\c{s} model for contractions.  

{\it Truncated Toeplitz operators} (TTO) are restrictions of usual Toeplitz operators onto $K_\theta$. 
Namely, given a space $K_\theta$ and $\psi \in L^\infty(\mathbb{T})$, define the operator
$A_\psi: K_\theta \to K_\theta$ by the formula 
$$
A_\psi f = P_\theta (\psi f), \qquad f\in K_\theta, 
$$
where $P_\theta$ stands for the orthogonal projection from $L^2(\T)$ onto $K_\theta$. 
Formally, $A_\psi$ depends also on $\theta$, but we always assume $\theta$ to be fixed and do not include it into notations.

Even if some special cases of TTO-s were well studied for a long time (e.g., $A_z$ -- the model operator 
of Nagy--Foia\c{s} theory and its functions $\varphi(A_z)  = A_\varphi$, $\varphi\in H^\infty$), 
a systematic study of TTO-s was initiated 
by a seminal paper by D. Sarason \cite{sar}. Note that one can consider operators with unbounded symbols (assuming there exists 
a bounded extension of an operator defined on the dense subset of bounded functions in $K_\theta$) 
and it may happen that a bounded TTO has no bounded symbol at all (see, e.g., \cite{barf}. In what 
follows we consider the case of bounded symbols only.

During the last decade the theory of truncated Toeplitz operators became an active field of research (see the surveys  
\cite{ob1, ob2} and the references therein). However, it seems that the spectral properties of TTO-s are far from being well understood.
E.g., to the best of our knowledge, it is not known whether the point spectrum of a TTO can have nonempty interior. 
For some classes of  TTO-s Fredholmness and invertibility criteria were found by 
M.\,C.~C\^{a}mara and J.\,R.~Partington \cite{cam}. In particular, they gave a description 
of the point spectrum of a TTO with a rational symbol in case of the Hardy space in the half-plane.

A problem which apparently did not attract much attention yet is to understand the dynamics of TTO-s and in particular 
their hypercyclicity. In fact, it is not known whether hypercyclic TTO-s do exist.

\begin{prob*} 
Do there exist hypercyclic truncated Toeplitz operators?
\end{prob*}

Note that, since $K_\theta$ is invariant with respect to antianalytic Toeplitz operators, 
one has $A_{\bar\varphi} = T_{\bar\varphi}$ for $\varphi\in H^\infty$. Also, recall that existence of 
an eigenvector for the adjoint operator is a trivial obstacle for hypercyclicity. Combining this, one comes to the following 
simple observation.

\begin{obs*}
If an inner function $\theta$ has a zero in $\D$, then for any $\varphi\in H^\infty$ 
the truncated Toeplitz operators $A_\varphi$ and $A_{\bar\varphi}$ are not hypercyclic.
\end{obs*}

\begin{proof}
If $\theta(\lambda) = 0$, then $k_\lambda(z) = (1-\bar \lambda z)^{-1}$ belongs to $K_\theta$, 
and $A_{\bar\varphi} k_\lambda = \overline{\varphi(\lambda)} k_\lambda$ whence $A_\varphi = A_{\bar\varphi}^*$ is not hypercyclic. 
Also, 
$$
A_\varphi \Big( \frac{\theta}{z-\lambda} \Big) = P_\theta \Big(   \frac{\varphi \theta}{z-\lambda}\Big) = 
 P_\theta \Big(   \theta \frac{\varphi -\varphi(\lambda)}{z-\lambda} +  \varphi(\lambda) \frac{\theta}{z-\lambda}\Big)  =
\varphi(\lambda) \frac{\theta}{z-\lambda}.
$$
Hence, $A_\varphi$ also has an eigenvector and so $A_{\bar\varphi}$ is not hypercyclic.
\end{proof}

Thus, if we are interested in hypercyclicity of TTO-s with analytic or antianalytic symbols, the only
interesting case is the case of singular inner function $\theta$. In this case the point spectrum is empty. 

\begin{prob*} 
Let $\theta$ be a singular inner function. Do there exist hypercyclic truncated Toeplitz operators 
with analytic or antianalytic symbols?
\end{prob*}

One can start with the case of a sufficiently regular symbol $\varphi$, say in the disk-algebra $A(\D)$ or even analytic in a neighborhood of 
$\overline{\D}$. In this case the spectrum of the operator $A_\varphi$ is known to be $\varphi(\sigma(\theta))$ \cite[Lecture III]{nik},
where $\sigma(\theta)$ stands for the boundary spectrum of $\theta$ (the set of those $\zeta\in \T$ for which 
$\liminf_{z\to\zeta, z\in \D} |\theta(z)|<1$). Thus, a necessary condition for hypercyclicity is that $\varphi(\sigma(\theta)) \cap\T \ne
\emptyset$. 

\begin{prob*} 
Let $\theta(z) = \exp\big(\frac{z+1}{z-1}\big)$  be the simplest atomic singular inner function. 
Do there exist hypercyclic truncated Toeplitz operators with a symbol $\varphi$ in $H^\infty$ or in $A(\D)$? This is equivalent to the following approximation problem: does there exist a function $f\in K_\theta$ such that 
$$
{\rm Clos}\,\{ \varphi^n f + \theta g:\ g\in H^2, n\in \mathbb{N}_0\} = H^2\,?
$$
\end{prob*}
\medskip

D. Sarason \cite{sar1} showed that the operator $(I+V)^{-1}$, where $V$
is the Volterra integration operator in $L^2(0,1)$, is unitarily equivalent to the operator 
$\frac{1}{2} (I+A_z)$ in $K_\theta$ with  $\theta(z) = \exp\big(\frac{z+1}{z-1}\big)$. 
In \cite{leon} it was shown that  $I+V$ (and, thus, $(I+V)^{-1}$) is not supercyclic. We conclude that
the truncated Toeplitz operator $\frac{1}{2} (I+A_z)$ is not supercyclic in $K_\theta$.
\bigskip


\section{Point spectrum of TTO-s with polynomial symbols}

In this section we consider truncated Toeplitz operators with symbols of the form
\begin{equation}
\label{poll}
\Phi(z) = \sum_{k=1}^N a_k z^{-k} + \sum_{l=0}^M c_l z^l, \qquad M,N\in \mathbb{N},  \ a_N, c_M \ne 0,
\end{equation}
and describe their point spectrum under some mild restrictions. It turns out that the set of eigenvectors $\lambda$ has a rather curious
structure involving some polynomial dependence between the values of $\theta$ at the roots of $\Phi - \lambda$. 
The next  theorem is very close to \cite[Theorem 5.4]{cam} where the invertibility criterion for
a TTO with a rational symbol is given in the case of the Hardy space in the half-plane.
The proof in \cite{cam} is based on the Riemann--Hilbert problem methods, while our approach is completely elementary.

Recall that the mapping $f \to \tilde f = \bar z\bar f \theta$ in $L^2(\T)$ is an involution on the space $K_\theta$ 
(considered as a subspace of $L^2(\T)$). Reproducing kernel of $K_\theta$ at the point $\lambda\in\D$ is 
given by
$$
k_\lambda^\theta (z) = P_\theta\Big(\frac{1}{1- \bar \lambda z}\Big) = 
\frac{1-\overline{\theta(\lambda)}\theta(z)}{1-\bar \lambda z}.
$$
Note that the conjugate kernel $\tilde k_\lambda^\theta = \widetilde{k_\lambda^\theta}$ is of the form
$$
\tilde k_\lambda^\theta (z)  = \frac{\theta(z) - \theta(\lambda)}{z -\lambda}.
$$

\begin{theorem}
\label{main1}
Let $\Phi$ be given by \eqref{poll} and let $z_j$, $j=1, \dots M+N$, be the zeros of $\Phi-\lambda$.
Assume that all $z_j$ are distinct and $|z_j| \ne 1$ 
for any $j$. Then $\lambda$ is an eigenvalue of $A_\Phi$ if and only if  there exist nonzero polynomials $P_1, P_2$
of degrees at most $M-1$ and $N-1$ respectively such that the following conditions hold:
\begin{equation}
\label{in}
z_j^N P_1(z_j) \theta(z_j) + P_2(z_j) = 0, \qquad |z_j| <1, 
\end{equation}
and
\begin{equation}
\label{out}
z_j^N P_1(z_j) +P_2(z_j) \overline{\theta(1/\overline z_j)} = 0, \qquad |z_j| >1. 
\end{equation}
In this case the corresponding eigenfunction is given by 
$$
f   = \sum_{|z_j| < 1} \beta_j    z_j^N P_1(z_j)  \tilde k^\theta_{z_j}
+ \sum_{|z_j| > 1} \beta_j P_2(z_j) k^\theta_{1/\bar z_j}
$$
for some complex coefficients $\beta_j$.
\end{theorem}

\begin{proof}
We are looking for solutions of the equation $A_\Phi f = \lambda f$, $f\in K_\theta$.
Denote by $\varphi$ the analytic part of $\Phi$, $\varphi(z) = \sum_{l=0}^M c_l z^l$. 
Note that $A_\Phi f = P_\theta (T_\Phi f)$, and 
$$
T_\Phi f = \varphi f + \sum_{k=1}^N a_k \frac{f - \sum_{j=0}^{k-1} \frac{f^{(j)} (0)}{j!} z^j}{z^k}.  
$$
Thus, $A_\Phi f = \lambda f$ is equivalent to the equation 
$$
\varphi f + \sum_{k=1}^N a_k \frac{f - \sum_{j=0}^{k-1} \frac{f^{(j)} (0)}{j!} z^j}{z^k} = \lambda f + \theta h,
$$
where $h\in H^2$ or, equivalently, 
$$
(\Phi-\lambda) f  = \theta h +  \sum_{k=1}^N a_k\sum_{j=0}^{k-1} \frac{f^{(j)} (0)}{j!} z^{j-k}.
$$
Multiplying by $z^N$, we get
\begin{equation}
\label{qq}
Qf = z^N \theta h + R,
\end{equation}
where $Q=  z^N (\Phi-\lambda)$ is a polynomial of degree $M+N$ with the zeros $z_j$ 
and $R(z)  = \sum_{k=1}^N a_k\sum_{j=0}^{k-1} \frac{f^{(j)} (0)}{j!} z^{N-k+j}$ 
is a polynomial of degree at most $N-1$. Since $z_j$ are distinct we can write
\begin{equation}
\label{go}
\frac{1}{Q(z)} = \sum_{j=1}^{M+N} \frac{\beta_j}{z-z_j}
\end{equation}
for some coefficients $\beta_j$. 

Assume that $\lambda$ is an eigenvector and so there exists a nontrivial $f \in K_\theta$ satisfying \eqref{qq}. 
If $|z_j| <1$, then the right-hand side of \eqref{qq} must vanish at $z_j$ and we get 
$z_j^N \theta(z_j) h(z_j) + R(z_j) = 0$. Hence, we have
\begin{equation}
\label{ff}
\begin{aligned}
f(z) & = \sum_{|z_j| < 1} \beta_j   \frac{z^N \theta(z) h(z) - z_j^N \theta(z_j) h(z_j)}{z - z_j} + 
\sum_{|z_j| < 1} \beta_j \frac{R(z) - R(z_j)}{z - z_j} \\
& + \sum_{|z_j|>1} \beta_j \frac{z^N \theta(z) h(z)}{z-z_j} + 
\sum_{|z_j|>1} \beta_j \frac{R(z)}{z-z_j}.
\end{aligned}
\end{equation}
For $|z_j|<1$ we further have
$$
\frac{z^N \theta(z) h(z) - z_j^N \theta(z_j) h(z_j)}{z - z_j} = \theta(z) \frac{z^N h(z) - z_j^N h(z_j)}{z - z_j} 
+ z_j^N h(z_j) \frac{\theta(z) - \theta(z_j)}{z - z_j}.
$$
Note that the first term is in $\theta H^2$, while the second belongs to $K_\theta$. 

Since the degree of $R$ is at most $N-1$, we have the Lagrange interpolation formula 
\begin{equation}
\label{lag}
\frac{R(z)}{Q(z)} =  \sum_{j=1}^{M+N} \beta_j \frac{R(z)}{z-z_j} = \sum_{j=1}^{M+N} \beta_j \frac{R(z_j)}{z-z_j}.
\end{equation}
Finally, note that for $|z_j| >1$ we have
$$
\frac{1}{z-z_j} =  \frac{\theta(z) \overline{\theta (1/\bar z_j)}}{z - z_j} + 
 \frac{1 - \theta(z)\overline{\theta(1/\bar z_j)}}{z- z_j},
$$
where, again, the first term is in $\theta H^2$, while the second is in $K_\theta$. Combining all these observations we 
conclude that $f=f_1+\theta f_2$ where $f_1\in K_\theta$ and $f_2 \in H^2$ are given by
\begin{equation}
\label{parts}
\begin{aligned}
f_1 (z) & = \sum_{|z_j| < 1} \beta_j    z_j^N h(z_j) \frac{\theta(z) - \theta(z_j)}{z - z_j} +
\sum_{|z_j| > 1} \beta_j R(z_j) \frac{1 - \theta(z)\overline{\theta(1/\bar z_j)}}{z- z_j}, \\ 
f_2 (z) & = \sum_{|z_j| < 1} \beta_j  \frac{z^N h(z) - z_j^N h(z_j)}{z - z_j}  
+   \sum_{|z_j| > 1} \beta_j \frac{R(z_j)\overline{\theta (1/\bar z_j)}}{z - z_j}.
\end{aligned}
\end{equation}
However, by our assumption, $f\in K_\theta$, whence $f_2 = 0$. In view of \eqref{go} this is equivalent 
to
\begin{equation}
\label{res}
\frac{z^N h(z)}{Q(z)} = \sum_{|z_j| < 1} \beta_j  \frac{z_j^N h(z_j)}{z - z_j}  
-   \sum_{|z_j| > 1} \beta_j \frac{R(z_j)\overline{\theta (1/\bar z_j)}}{z - z_j}.
\end{equation}
Hence $z^N h$ is a polynomial of degree at most $M+N-1$ and since $h$ is analytic at $0$ we conclude that
$h$ is a polynomial of degree at most $M-1$. If we put $P_1= h$ and $P_2 = R$, then
the condition $z_j^N \theta(z_j) h(z_j) + R(z_j) = 0$ for $|z_j| <1$ is equivalent to \eqref{in}.  
Comparing the residues at $z_j$ with $|z_j|>1$ in \eqref{res} we conclude that 
$z_j^N h(z_j) = - R(z_j)\overline{\theta (1/\bar z_j)}$ which is equivalent to \eqref{out}.
\medskip

To prove that conditions \eqref{in} and \eqref{out} are sufficient for $\lambda$ 
to be an eigenvalue for $A_\Phi$, we reverse the arguments. Assume that there exists $P_1$ and $P_2$ 
as in \eqref{in} and \eqref{out} and let $h= P_1$,  $R=P_2$. Then the interpolation formula \eqref{res} holds 
and so the function $f_2$ in \eqref{parts} is zero. Define $f_1$ by \eqref{parts}. 
Combining the formula for $f_1$ with \eqref{lag} we see that $f= f_1 = f_1+\theta f_2$ satisfies the equality \eqref{ff}. 
Using the fact that, $z_j^N \theta(z_j) h(z_j) + R(z_j) = 0$ for $|z_j| <1$ we, finally conclude that
$Qf = z^N h\theta + R$. Dividing by $z^N$ and comparing the coefficients at negative powers 
it is easy to see that $R$ must coincide with
$\sum_{k=1}^N a_k\sum_{j=0}^{k-1} \frac{f^{(j)} (0)}{j!} z^{N-k+j}$ where $a_k$ 
are the coefficients at $z^{-1}$ in \eqref{poll}.
\end{proof}

\begin{remark*}
{\rm Note that for some configurations of zeros with respect to the unit circle, 
the conditions \eqref{in} and \eqref{out} may be never satisfied. In particular, as is shown below, 
for a three-term TTO $\lambda$ is never  an eigenvalue if $z_1$ and $z_2$ lie in different components
of $\co\setminus\T$.}
\end{remark*}

\begin{remark*}
{\rm The above results remain true if there exist zeros of $\Phi-\lambda$ with $|z_j| =1$, but $z_j \notin \sigma(\theta)$ and so $\theta$
is analytic in a neighborhood of $z_j$. Such $z_j$ can be included in any of the conditions \eqref{in}
or \eqref{out} since in this case $\theta(z_j) = 1/\overline{\theta(1/\overline z_j)}$. 
The same is true if $z_j \in \sigma(\theta)$ but $|\theta(z_j)|=1$ and $\theta$ has a finite angular derivative at $z_j$ 
($z_j$ is a Julia--Carath\'eodory point for $\theta$)
and so the space $K_\theta$ contains the reproducing kernel at $z_j$. It is not, however, clear to us whether
$\lambda$ can be an eigenvalue if one of $z_j$ lies on $\T$ and 
is not a Julia--Carath\'eodory point for $\theta$. }
\end{remark*}
\bigskip


\section{Three-term truncated Toeplitz operators}

In this section we consider the case when $\Phi(z) =a \bar z +b + c z$, $z\in\T$. In this case 
the formulation of Theorem \ref{main1} will be substantially simplified. 

\begin{corollary}
\label{main2}
Let $\Phi(z) =\frac{a}{z} +b + c z $, $a,c\ne 0$, and assume that the zeros $z_1, z_2$ 
of the function $\Phi-\lambda$ are distinct and satisfy $|z_j| \ne 1$, $j=1,2$. 
Then $\lambda$ is an eigenvalue for $A_\Phi$ if and only if 
either 

1. $|z_1|, |z_2| <1$ and $z_1\theta(z_1) = z_2\theta (z_2)$, \\
or 

2. $|z_1|, |z_2| >1$ and $\frac{1}{z_1} \theta\Big(\frac{1}{\bar z_1}\Big) = \frac{1}{z_2} \theta\Big(\frac{1}{\bar z_2}\Big)$.

In case 1 the eigenfunction of $A_\Phi$ corresponding to $\lambda$ is given by 
\begin{equation}
\label{glj}
f_\lambda = z_1 \tilde k^\theta_{z_1} - z_2 \tilde k^\theta_{z_2},  
\end{equation}
while in case 2
$$
f_\lambda = \frac{1}{\bar z_1} k^\theta_{1/\bar z_1} - \frac{1}{\bar z_2}  k^\theta_{1/\bar z_2}.
$$
\end{corollary}

\begin{proof}
Since in this case $M=N=1$, the polynomials $P_1$ and $P_2$ in \eqref{in} and \eqref{out} are just constants. 
If both $z_1, z_2\in \D$, we have $P_1 z_1 \theta(z_1) + P_2 = P_1 z_2 \theta(z_2) + P_2 = 0$ which 
implies $z_1\theta(z_1) = z_2\theta (z_2)$. Analogously, 
if $z_1, z_2\in \Dh$, then 
$P_1 z_1 +P_2 \overline{\theta(1/\overline z_1)} = P_1 z_2 +P_2 \overline{\theta(1/\overline z_2)} = 0$,
whence $ z_1^{-1} \overline{\theta(1/\overline z_1)} = z_2^{-1} \overline{\theta(1/\overline z_2)}$.

Finally, consider the case $|z_1| <1$, $|z_2|>1$. Then $P_1 z_1 \theta(z_1) + P_2  
= 0 = P_1 z_2 +P_2 \overline{\theta(1/\overline z_2)}$, whence
$$
z_1\theta(z_1) \overline{\theta(1/\overline z_2)} = z_2,
$$
an impossible equality since its left-hand side in $\D$ and $|z_2|> 1$.

Conversely, equalities  $z_1\theta(z_1) = z_2\theta (z_2)$ and, respectively, 
$\frac{1}{\bar z_1} \theta\Big(\frac{1}{\bar z_1}\Big) = \frac{1}{\bar z_2} \theta\Big(\frac{1}{\bar z_2}\Big)$
imply the existence of the polynomials $P_1$ and $P_2$ of degree zero (nonzero constants)
such that \eqref{in} and\eqref{out} are satisfied and so $\lambda$ is an eigenvector by Theorem \ref{main1}.

The form of the eigenfunctions follows easily from the calculations of Theorem \ref{main1}.
\end{proof}

We excluded the case of multiple zeros of $\Phi - \lambda$ to avoid uninteresting technicalities. In this case our 
polynomial dependencies \eqref{in} and \eqref{out} will involve also the values of the derivatives of $\theta$
at the multiple zeros. However, in the case of symbols of the form $\Phi(z) = \frac{a}{z} +b + c z$ the answer is easy.

\begin{corollary}
\label{main2}
Let $\Phi(z) =\frac{a}{z} +b + c z $, $a,c\ne 0$, and assume that $\Phi-\lambda$ has a zero $z_0$ in $\D$ of multiplicity 2. 
Then $\lambda$ is an eigenvalue for $A_\Phi$ if and only if $z_0 \theta'(z_0) + \theta(z_0) =0$.
\end{corollary}

Note that if $z_1, z_2$ are zeros of $z(\Phi(z) - \lambda) = cz^2 + (b-\lambda)z +a$, then $z_1z_2 = a/c$. Put $\beta = a/c$.
If $|\beta|<1$ then $\lambda$ can be an eigenvalue only when $|z_1|, |z_2|<1$; if $|\beta|>1$, then a necessary condition
is that $|z_1|, |z_2|>1$. Finally, in the case when $|\beta| =1$, $\lambda$ is not an eigenvalue unless $|z_1|=|z_2|=1$. 

We now address the question about existence of $\lambda$, $z_1$ and $z_2$ satisfying the conditions
of Corollary \ref{main2}. We will show that under some (apparently, rather mild) restriction on the function $\theta$ 
we do not only have infinitely many solutions, but, moreover, the eigenvectors of a tridiagoanal TTO 
are complete in $K_\theta$. To formulate our results we will need some notions from the theory
of Smirnov classes $E^p(G)$ in general domains.


\subsection{Smirnov classes in multiconnected domains}
Let $G$ be a simply-connected domain in $\co$. 
Denote by $H^\infty(G)$ the class of all bounded analytic functions in $G$.
The {\it Smirnov class} $E^p(G)$, $0<p<\infty$, consists 
of those functions $f$ analytic in $G$ for which $f(\varphi(z)) (\varphi'(z))^{1/p} \in H^p$, where
$H^p$ is the usual Hardy space in $\D$ and $\varphi$ is some (any) conformal map of $\D$ onto $G$.
An equivalent definition describes $E^p(G)$ as the set of those functions $f$ for which there exists 
a sequence of rectifiable Jordan curves $\Gamma_n$ tending to the boundary (i.e. $\Gamma_n$ surrounds each compact subset 
of $G$ for sufficently large $n$) with the property 
$$
\sup_n \int_{\Gamma_n} |f(z)|^p |dz| <\infty.
$$
For the theory of Smirnov classes see, e.g., \cite[Chapter 10]{dur}.

A simply-connected domain  $G$ with rectifiable boundary is said to be a {\it Smirnov domain}  if $\varphi' $ is an outer function 
(as above,  $\varphi$ is some conformal map of $\D$ onto $G$). In this case we will say that a function $f\in E^p(G)$ 
is {\it outer} if $f(\varphi(z)) (\varphi'(z))^{1/p}$ is an outer function in $H^p$ and we will say that
$f$ {\it has no singular inner factor} if  $f(\varphi(z)) (\varphi'(z))^{1/p} = B(z) F(z)$, where $B$ is a Blaschke product
and $F$ is an outer function in $\D$.

We will need to define similar objects for the case of an annulus. For general multiconnected domains $G$
factorization theory in $E^p(G)$ was developed by D. Khavinson in \cite{khav} (see, also, \cite{khav1}). We will
however use the following  simple equivalent description of these classes. For $\beta \in \D$, $\beta\ne 0$, 
consider the annulus
$$
R_\beta =\{|\beta| < |z|<1\}.
$$
For $\alpha\in \T$ let $R_\beta^\alpha = R_\beta\setminus \{r\alpha: \ |\beta| < r<1\} $ be the annulus with a cut. 
Then $f\in E^p(R_\beta)$ if and only if for any $\alpha\in \T$ and a conformal mapping $\varphi: \D\to R_\beta^\alpha$
the function $f(\varphi(z)) (\varphi'(z))^{1/p}$ is in $H^p$. Analogously, we say that
$f\in E^p(R_\beta)$ has no singular inner factor in $R_\beta$ if  $f(\varphi(z)) (\varphi'(z))^{1/p}$ has no singular inner factor 
in $\D$ for any $\alpha$ and $\varphi$. Obviously, it is sufficient to take only two different values of $\alpha$.


\subsection{Completeness of eigenvectors of a tridiagonal TTO}
Let $\Phi(z) =\frac{a}{z} +b + c z $, $a,c\ne 0$, and let  $\beta = a/c$.
If $|\beta|<1$ we put
\begin{equation}
\label{ro1}
\Psi(z) = z\theta(z) -\frac{\beta}{z} \theta\Big(\frac{\beta}{z}\Big),
\end{equation}
while for $|\beta| >1$ put
\begin{equation}
\label{ro2}
\Psi(z) = z\theta(z) -\frac{1}{\bar \beta z} \theta\Big(\frac{1}{\bar \beta z}\Big).
\end{equation}
As before, let $R_\beta = \{z: \ |\beta| < |z| <1\}$ for $|\beta| <1$ and put
$R_\beta = \{z: \ |\beta|^{-1} < |z| <1\}$ for $|\beta| >1$. Then the function $\Psi$ is analytic
in the respective choice of the annulus $R_\beta$ and, moreover, $\Psi \in H^\infty(R_\beta)$.

\begin{theorem}
\label{main3}
Let $\Phi(z) =\frac{a}{z} +b + c z $, $a,c\ne 0$, and $|\beta| = |c/a| \ne 1$. Assume that the function
$\Psi$ defined by \eqref{ro1} or \eqref{ro2} has no singular inner factor in $R_\beta$. Then
the set of eigenvectors of $A_\Phi$ is complete in $K_\theta$. 
\end{theorem}

\begin{proof}
We consider the case  $|\beta| <1$, the case $|\beta|>1$ is analogous. 
Assume that $f\in K_\theta$ is orthogonal to all eigenfunctions $f_\lambda$ given by \eqref{glj}.
Note that $(\tilde f, \tilde g) = (g,f) $ for $f, g\in K_\theta$. 
Then we have
$$
0 =(f_\lambda, f) = z_1 (\tilde k^\theta_{z_1}, f) -  z_2 (\tilde k^\theta_{z_2}, f) = 
z_1 (\tilde f, k^\theta_{z_1}) - z_2 (\tilde f, k^\theta_{z_2}) = z_1 \tilde f(z_1) - z_2 \tilde f(z_2).
$$
Recall that $z_1z_2 = \beta$ and $z_1\theta(z_1) = z_2\theta(z_2)$. Consider the function 
$$
F(z) =  z \tilde f(z) -\frac{\beta}{z} \tilde f \Big(\frac{\beta}{z}\Big).
$$
Then we conclude that $F(z) = 0$ whenever $z \in R_\beta$ and $\Psi(z) =0$, whence 
$F/\Psi$ is analytic in $R_\beta$. This is true even if $\Psi$ has a multiple zero, since in this case it is not difficult to show
that $F$ will have the zero of at least the same multiplicity. 

It is clear that $F\in E^2(R_\beta)$. Let us show, using the assumption that $\Psi$
has no singular inner factor in $R_\beta$, that $F/\Psi \in E^2(R_\beta)$. 
First of all note that for almost all $z\in \T$
$$
|\Psi(z)| \ge |\theta(z)| - |\beta| = 1 - |\beta|,
$$
a similar estimate holds for $|z| = |\beta|$.

Fix some $\alpha\in\T$ such that $|\theta(r\alpha)|\to 1$ as $r\to 1-$ and 
$\big|\theta\big(\frac{\beta}{r\alpha} \big)\big|\to 1$ as $r\to |\beta|+$, and
consider the annulus with a cut $R_\beta^\alpha$. For a conformal map $\varphi$ of 
$\D$ onto $R_\beta^\alpha$ we have $(F\circ \varphi) \cdot (\varphi')^{1/2} \in H^2$ and so $(F\circ \varphi) \cdot (\varphi')^{1/2} = I F_o $
where $I$ is an inner function and $F_o$ is an outer function. Also, $\Psi\circ \varphi = B \Psi_o$ for a Blaschke product
$B$ and an outer function $\Psi_o$. Since $F/\Psi$ is analytic in $R_\beta^\alpha$ we conclude that $B$ divides $I$. 
Thus,
$$ 
\frac{(F\circ \varphi)}{\Psi\circ \varphi} \cdot (\varphi')^{1/2} = JH
$$
for some inner function $J$ and an outer function $H$. Now it is sufficient to show that
$H\in L^2(\T)$. We have $1/\Psi \in L^\infty(\partial R_\beta)$. By our choice of $\alpha$ 
there exists $\varepsilon, \delta>0$ such that $|\Psi(r\alpha)| \ge \varepsilon$ when $r\in (|\beta|, |\beta| +\delta)
\cup(1-\delta, 1)$. Since $(F\circ \varphi) \cdot (\varphi')^{1/2} \in L^2(\partial R_\beta^\alpha)$ 
we have 
$$ 
H\in L^2 \big( \varphi^{-1} (\partial R_\beta \cup \{ r\alpha: r \in (|\beta|, |\beta| +\delta)
\cup(1-\delta, 1) \})\big). 
$$
The functions $(F/\Psi)\circ \varphi $ and $\varphi'$ are obviously bounded on 
$\varphi^{-1} \big(\{ r\alpha: r \in (|\beta| +\delta, 1-\delta) \}\big)$. Thus, $H\in H^2$ and we conclude that 
$F/\Psi \in E^2(R^\alpha_\beta)$ for almost all $\alpha$, and so $F/\Psi \in E^2(R_\beta)$.

Recall that any function $f\in K_\theta$ has a meromorhic pseudocontinuation to $\Dh = \{|z| >1\}$ such that its nontangential 
boundary values on $\T$ taken from $\Dh$ coincide  with its boundary values in $\D$ and also
$f/\theta$ belongs to the Hardy space $H^2(\Dh)$ (which is the same as $E^2(\Dh)$). We claim that 
$F/\Psi \in H^2(\Dh)$. Indeed, 
$$
\frac{F(z)}{\Psi(z)} = \dfrac{\frac{f(z)}{\theta(z)} - \frac{\beta}{z^2\theta(z)} f \Big(\frac{\beta}{z}\Big)}
{1 -  \frac{\beta}{z^2\theta(z)} \theta \Big(\frac{\beta}{z}\Big)},
$$
and the function in the denominator is bounded away from zero. Analogously, one shows that 
$F/\Psi \in H^2(\{|z| <|\beta|\})$. Note also that $F(z)/\Psi(z) \to 0$  as $|z|\to \infty$ or $|z| \to 0$.

The nontangential boundary values of $F/\Psi$ taken from inside of $R_\beta$ coincide with the boundary values
taken from the outer domains $\{|z|<|\beta|\}$ and $\{|z| >1\}$ almost everywhere. 
Applying the Cauchy formula over the contours approaching the boundary 
from the opposite sides and passing to the limit (which is possible for functions in the classes $E^2$) 
we see that $F/\Psi$ has an analytic extension across the circles
$\{|z| = |\beta|\}$ and $\{|z| =1\}$. Thus, $F/\Psi$ is an entire function tending to zero at infinity.
We conclude that $F/\Psi \equiv 0$, whence $f\equiv 0$.
\end{proof}

We see that in the conditions of Theorem \ref{main3} the operator $A_\Phi$ has many eigenvectors whence 
$A_{\bar \Phi}$ is not hypercyclic.
\medskip

Now let us consider in more detail the case when  the boundary spectrum of $\theta$ is one point, e.g., 
$\sigma(\theta) = \{1\}$. In this case $\theta$ is meromorphic in 
$\co\setminus\{1\}$, while the function $\Psi$ will be meromorphic in $\co\setminus\{1, \beta\}$
(here we again consider the case $|\beta| <1$). Let $R_\beta^\alpha$ be the annulus with a cut which does not touch
$1$ and $\beta$, and let $\varphi$ 
be the corresponding conformal map. Then the function $\Psi\circ \varphi $ in $\D $ can have no singular inner factors 
except the atomic singular functions $\exp\big(a_j \frac{z+\zeta_j}{z-\zeta_j}\big)$, $j=1,2$, where $a_j>0$ 
and $\zeta_j\in \T$ are such that $\varphi(\zeta_1) =1$ and 
$\varphi(\zeta_2) =\beta$. Indeed, if $z\in \T$, $z\ne \zeta_1, \zeta_2$, and $z$ belongs to the support of the singular measure
from the inner factor of $\Psi\circ \varphi $, then there exists a sequence $z_n\to z$, $z_n\in \D$,
such that $\Psi(\varphi(z_n)) =o((z_n-z)^m)$ for any $m\in\N$, which contradicts the fact that $\Psi$ is analytic at $\varphi(z)$
(note that the conformal map $\varphi$ is itself analytic in a neighborhood of $1$ and $\beta$). 

Therefore, it is easy to see that $\Psi$ has no singular inner factor in $R_\beta$ if and only if 
\begin{equation}
\label{rep}
\limsup_{r\to 1-} (1-r)\log |\Psi(r)| = 0.
\end{equation}

\begin{prob*}
Is it true that the function $\Psi(z) = z\theta(z) -\frac{\beta}{z} \theta\Big(\frac{\beta}{z}\Big)$
has no singular inner factor in $R_\beta$ for any inner function $\theta$?
\end{prob*}

We believe that the answer to this question is positive, at least, for a wide class of inner functions. 

In the case when $\sigma(\theta) = \{1\}$ one can make several observations. They show that
the case  of a (possible) singular inner factor in the factorization of $\Psi$ is exceptional and can happen only rarely. 
Also, the assumption that $\Psi$ has a singular inner factor implies important restrictions on the zeros of $\theta$
and prohibits them to approach to 1 nontangentially.

\begin{corollary}
\label{main4}
Let $\theta$ be an inner function such that  $\sigma(\theta) = \{1\}$. Then
\medskip

1. For any $a\in \co$, $a\ne 0$, there exist at most one value of $c\in \co$, $c\ne 0$, $|c| \ne |a|$,
such that the function $\Psi$ given by \eqref{ro1} or \eqref{ro2} has a singular inner factor. 

Thus, for all values of $c\in \co$, $c\ne 0$, $|c| \ne |a|$, except at most one, the set of eigenvectors
of the operator $A_\Phi$ is complete and the operator $A_\Phi$ is not 
\medskip
hypercyclic.

2. If $\theta$ has an atomic singular factor $\exp\big(a \frac{z+1}{z-1}\big)$, $a>0$, 
then $\Psi$ given by \eqref{ro1} or \eqref{ro2} has no singular inner factor for any $\beta\ne 0$ with $|\beta| \ne 1$.
\medskip

3. If for some $\beta$ the function  $\Psi$ has a singular inner factor, then  
there exist $a>0$, $m\in \N$ and $c>0$ such that
all zeros of $\theta$, maybe except of a finite number, lie in the domain
$$
\Big\{ z\in \D: \ \Big|\exp\Big(a \frac{1+z}{1-z}\Big)\Big| \ge c|z-1|^m \Big\}.
$$
\end{corollary}

\begin{proof}
1. Assume that there exist $\beta_1, \beta_2 \ne 0$, $|\beta_1|, |\beta_2| \ne 1$ and $\beta_1\ne \beta_2$ such that
$\Psi_{\beta_1}(z) = z\theta(z) - \frac{\beta_1}{z} \theta\Big(\frac{\beta_1}{z}\Big)$
and $\Psi_{\beta_2}$ has singular inner factors. Then, by \eqref{rep},
$$
|\Psi_{\beta_j}(r)| \le \exp\Big(- \frac{a}{1-r}\Big)\qquad \text{as} \ r\to 1-,
$$
for some $a>0$, $j=1,2$. Then the function
$$
\frac{\beta_1}{z} \theta\Big(\frac{\beta_1}{z}\Big) - \frac{\beta_2}{z} \theta\Big(\frac{\beta_2}{z}\Big)
$$
tends to zero faster than any power as $z=r \to 1-$. Since this function is analytic at $z=1$, it is identically zero, which easily leads
to a contradiction. 

Now, if $a$ is fixed there exists at most one $c$ as above such that  $\Psi_{\beta}$ has a singular inner factor
for $\beta = c/a$. Also, $A_{\bar \Phi}$ has a complete set of eigenvectors for all values of $c$ except at most one.

Statement 2 follows by the same argument as Statement 1. To prove Statement 3, fix a a conformal mapping $\varphi$
of the  annulus $R_\beta^\alpha$ with a cut at $\alpha\ne 1$ such that $\varphi(1) =1$. Then, for some $a, \tilde a>0$, we have
$$
|\Psi(z)| \le \Big| \exp\Big( a \frac{\varphi^{-1}(z) +1}{\varphi^{-1}(z) -1}\Big) \Big| \le  
\Big| \exp\Big(\tilde a \frac{z +1}{z -1}\Big) \Big|
$$
as $z\to 1$ inside some Stolz angle at 1;
we use here that $\varphi$ is analytic in a neighborghood of 1 and preserves the angles. 
On other hand, $\Psi (z) = \frac{\beta}{z} \theta\big(\frac{\beta}{z}\big)$ at the points where
$\theta(z) =0$. However, $ \frac{\beta}{z} \theta\big(\frac{\beta}{z}\big) \sim C(z-1)^m$ for some $C\ne 0$ 
and $m\in \N_0$ due to analyticity at $z=1$.
\end{proof}

It follows from Corollary \ref{main4} that
in the case when $\sigma(\theta) = \{1\}$ and $\theta$ has an atomic singular factor 
or infinitely many zeros approaching the point 1 nontangentially all tridiagonal operators $A_\Phi$ 
with $|a| \ne |c|$ have complete sets of eigenvectors.
\bigskip
\\
\noindent
{\bf Acknowledgements. } The authors are grateful to Dmitry Khavinson, 
Jonathan Partington and Ilia Zlotnikov for useful discussions and helpful remarks.

\end{document}